\documentclass{article}
\usepackage[utf8]{inputenc}
\usepackage{BrandonStylePlusBib/BAEOstyle}

\title{Harmonic Analysis and Statistics of the First Galois Cohomology Group}
\author{Brandon Alberts \\ \href{mailto:balbert1@emich.edu}{balbert1@emich.edu}
\and Evan O'Dorney \\ \href{mailto:emo916math@gmail.com}{emo916math@gmail.com}
}
\date{\today}

\begin{document}

\maketitle

\begin{abstract}
We utilize harmonic analytic tools to count the number of elements of the Galois cohomology group $f\in H^1(K,T)$ with discriminant-like invariant $\inv(f)\le X$ as $X\to\infty$. Specifically, Poisson summation produces a canonical decomposition for the corresponding generating series as a sum of Euler products for a very general counting problem. This type of decomposition is exactly what is needed to compute asymptotic growth rates using a Tauberian theorem. These new techniques allow for the removal of certain obstructions to known results and answer some outstanding questions on the generalized version of Malle's conjecture for the first Galois cohomology group.
\end{abstract}

\section{Introduction}
Let $K$ be a number field. Extensions of $K$ with certain Galois groups and fixed resolvent can be parametrized by coclasses in the first cohomology group $H^1(K,T)$ of a Galois module $T$. It is natural to count such coclasses asymptotically by an invariant such as the norm of the discriminant or the product of the ramified primes.

The first author \cite{alberts2019} proved a result counting coclasses by discriminant-like invariants which are \textbf{Frobenian;} that is, they are determined at almost all places $p$ of $K$ by the Frobenius element $\Leg{F/K}{p}$ for some finite extension $F/K$ (we call the finitely many places where this is not satisfied \textbf{irregular places}). Moreover, the result also applies if we impose a local condition $L_p \subseteq H^1(K_p, T)$ at each place $p$, provided that these conditions are again Frobenian in the appropriate sense. Explicitly, this result can be interpreted as giving the asymptotic size of an ``infinite Selmer group", where for any family of local conditions $\mathcal{L}=(L_p)$ we define
\[
H^1_{\mathcal{L}}(K,T) = \{f\in H^1(K,T) : \res_p(f)\in L_p\}\,.
\]
If we allow $L_p$ to include non-trivial ramification at infinitely many places, this set may be infinite (for example, let $L_p=H^1(K_p,T)$ for all $p$ so that $H^1_\mathcal{L}(K,T) = H^1(K,T)$). We can still say something about the ``size" by talking about how the 1-coclasses are distributed with respect to some ordering. Given an admissible ordering $\inv : H^1(K,T) \to I_K$ valued in the group of fractional ideals (as in \cite{alberts2019}), we can define
\[
H^1_{\mathcal{L},\inv}(K,T;X) = \{f\in H^1_{\mathcal{L}}(K,T) : \mathcal{N}_{K/\Q}(\inv(f))<X\}\,,
\]
which is a finite set, and we can ask how this set grows as $X\to \infty$. We will often omit the ``$\inv$" from the subscript when it is clear from context.

We state the Asymptotic Wiles result proven in \cite{alberts2019} for comparison:
\begin{theorem*}[Asymptotic Wiles Theorem; Theorem 4.7 \cite{alberts2019}]
Let $T$ be a finite Galois $K$-module, and $\mathcal{L}=(L_p)_p$ and $\inv$ a family of local conditions and admissible ordering respectively. Suppose
\begin{enumerate}[(a)]
    \item $\mathcal{L}$ and $\inv$ are Frobenian;
    \item $L_p\le H^1(K_p,T)$ is a subgroup and, for almost all places $p$, $L_p$ contains $H^1_{ur}(K_p,T)$. For infinitely many places, $L_p$ is strictly larger than $H^1_{ur}(K_p,T)$;
    \item\label{it:div} For all places $p$, if $f,f'\in H^1(K_p,T)$ such that $\langle f|_{I_p}\rangle =\langle f'|_{I_p}\rangle\le H^1(I_p,T)$ then $\nu_p(\inv(f))=\nu_p(\inv(f'))$ (where $\nu_p(\mathfrak{a})$ equals the order to which $p$ divides $\mathfrak{a}$).
\end{enumerate}
Then
\[
    |H^1_{\mathcal{L}}(K,T;X)| \sim c_\inv(K,\mathcal{L}) X^{1/a_\inv(\mathcal{L})}(\log X)^{b_\inv(\mathcal{L})-1}
\]
for explicit positive integers $a_{\inv}(\mathcal{L})$ and $b_{\inv}(\mathcal{L})$, and an explicit positive real number $c_{\inv}(K,\mathcal{L})$.
\end{theorem*}

The method of proof in \cite{alberts2019} consists of using the Greenberg-Wiles identity to rewrite the Dirichlet series counting such coclasses as a finite sum of Frobenian Euler products, and then applying a Tauberian theorem to extract information about the average size of the coefficients. Due to the close relationship to the Greenberg-Wiles identity, and viewing $|H^1_{\mathcal{L}}(K,T;X)|$ as the asymptotic size of an ``infinite Selmer group", this result is dubbed the Asymptotic Wiles Theorem.

Given a compact box (i.e. some family of local conditions $\mathcal{L}=(L_p)$ for which $L_p$ is a subgroup of $H^1(K_p,T)$ at all places and $L_p=H^1_{ur}(K_p,T)$ at all but finitely many places), the Greenberg-Wiles identity acts as a local-to-global principle expressing $|H^1_{\mathcal{L}}(K,T)|$ as a product of local factors. Since Greenberg-Wiles requires the $L_p$ to be \emph{subgroups}, some adjusting is necessary to get information about more general subsets such as $H^1_{\mathcal{L}}(K, T; X)$. In particular, there is a need for hypothesis \ref{it:div}, which we might call the property of being \textbf{constant on divisions}, where a \textbf{division} of a group $G$ is a minimal subset $C\subset G$ closed under conjugation (i.e. $g\in C$ implies $g^h\in C$) and under invertible powers (i.e. if $n$ is coprime to the order of $G$ and $g\in C$, then $g^n \in C$). In particular, \ref{it:div} implies that $\nu_p\inv:H^1(K_p,T) \to \Z$ is constant on divisions of the group $H^1(K_p,T)$ and is equivalent to requiring $\nu_p(\inv f)$ to depend only on the subgroups of $H^1(I_p, T)$ in which $f|_{I_p}$ lies for each $p$.

In this paper, we use harmonic analysis on adelic cohomology as a replacement for the Greenberg-Wiles identity, modeled after the celebrated use of harmonic analysis on the adeles in Tate's thesis \cite{tate1950}. Harmonic analysis on adelic cohomology has been used by Frei-Loughran-Newton \cite{frei-loughran-newton2018,frei-loughran-newton2019} as an alternate approach to counting abelian extensions of number fields with prescribed local conditions, and by the second author \cite{odorney2021reflection} to give a streamlined proof of the Ohno-Nakagawa reflection theorem and its generalizations. Poisson summation, in this setting, is a suitable generalization of Greenberg-Wiles that will allow us to lift some of the hypotheses in the main theorem to prove the following extended version of the Asymptotic Wiles theorem:
\begin{theorem}\label{thm:comparisonthm}\label{thm:correctedmain}
Let $T$ be a finite Galois $K$-module, and $\mathcal{L}=(L_p)_p$ and $\inv$ a family of local conditions and admissible ordering respectively. Suppose
\begin{enumerate}[(a)]
    \item $\mathcal{L}$ and $\inv$ are Frobenian.
    \item For all but finitely many places $p$, $L_p$ is a union of cosets of $H^1_{ur}(K_p,T)$ containing the identity coset.
    \item\label{it:GW} The conditions $L_p$ are viable, that is, $H^1_{\mathcal{L}}(K, T) \neq \emptyset$.
\end{enumerate}
Then
\[
    \size{H^1_{\mathcal{L}}(K,T;X)} \sim c_\inv(K,\mathcal{L}) X^{1/a_\inv(\mathcal{L})}(\log X)^{b_\inv(\mathcal{L})-1}
\]
for explicit positive integers $a_{\inv}(\mathcal{L})$ and $b_{\inv}(\mathcal{L})$, and an explicit positive real number $c_{\inv}(K,\mathcal{L})$.
\end{theorem}

This is a significant improvement, in particular weakening (b) by allowing $L_p$ to be non-subgroups and completely removing the constancy-on-divisions hypothesis \ref{it:div} from the original result. Previously studied methods for counting number fields ordered by an arbitrary invariant (such as Wright's work on abelian extensions \cite{wright1989}) do not require a constancy-on-divisions assumption, suggesting that it is merely an artifact of the methods being used in \cite{alberts2019}. By using harmonic analysis to bypass this condition, we verify this intuition.

Instead we have a simple viability condition (\ref{it:GW} of Theorem \ref{thm:comparisonthm}), which is needed in order to avoid pathologies connected with the Grunwald--Wang theorem. For example, $H^1(\Q, C_8)$ contains no elements that are totally inert at $2$, so imposing a local condition $L_2 \subset H^1(\Q_2, C_8)$ consisting only of inert coclasses would give us no global elements to count at all. This issue will be discussed at greater length in Section \ref{sec:viable}.

As a particular example, if $G \subseteq S_n$ is a faithful transitive representation of $G$, the discriminant $\disc(f)$ of the degree-$n$ \'etale algebra attached to a coclass $f \in H^1(K,G)$ is a natural invariant to order coclasses in $H^1(K,G)$ by. (A more in-depth discussion of discriminants of \'etale algebras is found in \cite[Definition 3.2]{alberts2019}.) It satisfies the constancy-on-divisions property at all but finitely many places, but some issues occur at wildly ramified places. The first author uses this to prove upper and lower bounds on $|H^1_{\mathcal{L}}(K,T;X)|$ of the same order of magnitude in \cite{alberts2019} when $T\normal G$ has an induced Galois module structure given by $x.t=\pi(x)t\pi(x)^{-1}$ for some $\pi:G_K\to G$, and remarks that constancy on divisions is the obstruction to producing the asymptotic main term (in this setting $H^1(K,T)\to H^1(K,G)$ has finite kernel, and thus $H^1(K,T)$ can still be ordered by the discriminant of the \'etale algebra corresponding to the image in $H^1(K,G)$). Theorem \ref{thm:comparisonthm} is a sufficient generalization to prove the following corollary, which counts nonabelian extensions with specified nonabelian Galois group $G$ and specified resolvent by an abelian normal subgroup. In the language of \cite{alberts2019}, this counts $(T\normal G)$-towers ordered by discriminant, where $T$ is an abelian normal subgroup of $G$.
\begin{corollary}\label{cor:surj}
Let $T$ be an abelian normal subgroup of a finite group $G$, and $\pi:\Gal(\overline{K}/K)\rightarrow G$ a homomorphism with $T\pi(\Gal(\overline{K}/K))=G$ (or equivalently $\pi$ surjects onto $G/T$), defining a Galois action on $T$ by $x.t=\pi(x)t\pi(x)^{-1}$.
\begin{itemize}
\item[(i)]{If $\mathcal{L}$ and $\inv$ satisfy the hypotheses of Theorem \ref{thm:comparisonthm} and $S$ is the set of irregular places, then the limit
\[
\lim_{X\to\infty}\frac{\Size{\{f\in H^1_{\mathcal{L}}(K,T;X) : f*\pi\text{ surjective}\}}}{|H^1_{\mathcal{L}}(K,T;X)|}
\]
converges, where $(f*\pi)(x)=f(x)\pi(x)$ is understood to apply to a representative of $f$ in $Z^1(K,T)$. Moreover, the limit is
\begin{itemize}
    \item[(a)]{positive if $\pi$ is surjective,}
    \item[(b)]{positive if $T=\langle f_p(I_p) : f_p\in L_p,p\not\in S\rangle$, and}
    \item[(c)]{equal to $1$ if $T=\langle f_p(I_p) : f_p\in L_p^{[a_\inv(\mathcal{L})]},p\not\in S\rangle$ where $L_p^{[m]}=\{f\in L_p : \nu_p(\inv(f)) = m\}.$}
\end{itemize}
}

\item[(ii)]{
If $G\subset S_n$ is a transitive representation of $G$, then the invariant
\[
\disc_{\pi}(f) = \disc(f*\pi)
\]
satisfies the hypotheses of Theorem \ref{thm:comparisonthm}, where $\disc:\Hom(\Gal(\overline{K}/K),S_n)\to I_K$ is the discriminant on the associated \'etale algebra of degree $n$.
}
\end{itemize}
\end{corollary}

Corollary \ref{cor:surj} is a generalization of \cite[Theorem 5.3]{alberts2019}, and for the most part the proof is the same after replacing the Asymptotic Wiles' Theorem in \cite{alberts2019} with the more powerful Theorem \ref{thm:comparisonthm}. In particular, as described in \cite{alberts2019}, $\disc_\pi$ satisfies condition $(i)(b)$ so that the number of towers with bounded discriminant
\[
N(L/K,T\normal G;X) = \Size{\{f\in H^1_{\disc_\pi}(K,T;X): f*\pi\text{ surjective}\}}
\]
is asymptotic to
\[
c X^{1/a(T)} (\log X)^{b(K,T(\pi))-1}
\]
for some positive constant $c$ as $X\to \infty$. This proves the generalization of Malle's conjecture to towers for $T$ abelian described in \cite{alberts2019} with an exact asymptotic.

In addition to allowing us to relax the hypotheses and prove the asymptotic main terms, the harmonic analytic method has noteworthy elegance. The proof found in this note is many times shorter than the original proof in \cite{alberts2019}, and it naturally motivates the resulting finite sum of Euler products as the sum over the dual arising in a Poisson summation.

\section{Fourier analysis on adelic cohomology}

Our goal is to study the average value of certain functions over the global cohomology classes $H^1(K,T)$. Given such a function $w$, we are looking to understand the summation
\[
\sum_{f\in H^1(K,T)} w(f)\,.
\]
This can appear as an average of an arithmetic function if $w(f)=w_X(f)$ is zero when $\inv(f)\ge X$, so that we can consider the limiting behavior as $X\to \infty$. It can also appear as a Dirichlet series, such as when $w(f) = \inv(f)^{-s}$, of which we are interested in studying the meromorphic continuation. (Note that in many cases these two examples are intimately related, and results for one can imply results for the other.)

The idea is to realize $H^1(K,T)$ as a discrete subgroup of some locally compact group, so that $\sum w(f)$ appears to be one side of a Poisson summation. We recall the Poisson formula for locally compact abelian groups (see, for instance, Hewitt and Ross \cite{hewitt-ross1970}, p.~246):

\begin{theorem} \label{thm:Poisson-0}
Let $G$ be a locally compact abelian group and $w : G \to \C$ a continuous function of class $L^1$. Suppose that $G_0 \subseteq G$ is a closed subgroup such that
\begin{enumerate}[$($a$)$]
  \item\label{it:poi-a-w'} For each $x \in G$, the integral
  \[
    w'(x) = \int_{y \in G_0} w(x + y)\, dy
  \]
  is absolutely convergent;
  \item\label{it:poi-b-w'} The resulting function $w' : G/G_0 \to \C$ is $L^1$;
  \item\label{it:poi-c-fo} Its Fourier transform $\widehat{w'}$ is $L^1$ on $(G/G_0)^{\vee} \cong G^\vee/G_0^\perp$ (where $G_0^\perp$ is the annihilator of $G_0$).
\end{enumerate}
Then, with respect to an appropriate scaling of the measures on $G_0$ and $G_0^\perp$,
\[
  \int_{x \in G_0} w(x)\, dx = \int_{y \in G_0^\perp} \hat w(y)\, dy.
\]
\end{theorem}

In our case, we will realize $G_0=H^1(K,T)$ as a discrete group, so that
\[
\int_{f\in H^1(K,T)} w(f)\, df = \sum_{f\in H^1(K,T)} w(f)
\]
(up to scaling the measure). 

The goal of this section is to set up the other features of Poisson summation in this setting, including the definition of $G$ and the corresponding groups $G^\vee$ and $G_0^{\perp}$.

\begin{definition}
The \textbf{adelic cohomology} of a Galois $K$-module $T$ is the restricted direct product
\[
  H^1(\A_K, T) \coloneqq \sideset{}{'}\prod_p H^1(K_p, T)_{H^1_{ur}(K_p, T)}\,,
\]
or more explicitly
\[
H^1(\A_K,T) \coloneqq \bigg\{ (f_p)\in \prod_p H^1(K_p,T) \bigg\vert f_p\in H^1_{ur}(K_p,T) \text{ for all but finitely many }p\bigg\}.
\]
\end{definition}
When each $H^1(K_p,T)$ is given the discrete topology, $H^1(\A_K,T)$ is a locally compact abelian group. Its Pontryagin dual, by local Tate duality, is isomorphic to $H^1(\A_K, T^*)$ where $T^* = \Hom(T, \mu)$ is the Tate dual of $T$.

There exists a map
\[
  i_T : H^1(K, T) \to H^1(\A_K, T)
\]
coming from a choice of embeddings $\bar K \hookrightarrow \bar K_p$, which we fix throughout. We gather the necessary information on this map from the Poitou-Tate nine-term exact sequence \cite[Theorem 8.6.10]{neukirch-schmidt-wingberg2008}, which in the case that $S$ is the set of all places is given by
\[
\begin{tikzcd}
0 \rar & H^0(K,T) \rar & H^0(\A_K,T) \rar & H^2(K,T^*)^{\vee} \rar &{}\\
\rar & H^1(K,T) \rar{i_T} & H^1(\A_K,T) \rar{(i_{T^*})^\vee} & H^1(K,T^*)^\vee \rar &{}\\
\rar & H^2(K,T) \rar & H^2(\A_K,T) \rar & H^0(K,T^*)^\vee\rar & 0,
\end{tikzcd}
\]
where each of the homomorphisms are continuous and the nine terms have the following topologies respectively: finite, compact, compact, discrete, locally compact, compact, discrete, discrete, finite.

Although $i_T$ need not be injective, we see that the kernel is the continuous image of a compact group into a discrete cocompact group, and so is necessarily finite. Moreover, the middle row shows that the images of $i_T$ and $i_{T^*}$ are dual lattices in their respective cohomology groups. We thus derive the following corollary of Poisson summation:
\begin{theorem} \label{thm:Poisson}
Normalize the measure on $H^1(K_p, T)$ so that each single point has measure
\[
  \frac{1}{\size{H^0(K_p, T)}},
\]
and take the product of these measures as a measure on $H^1(\A_K, T)$ so that the compact subgroup
\[
  \prod_{p\text{ finite}} H^1_{ur}(K_p, T)
\]
has measure
\[
  \prod_{p \mid \infty} \frac{1}{\size{H^0(K, M)}}.
\]
Then for any function $w:H^1(\A_K,T)\rightarrow \C$ satisfying the hypotheses of Theorem \ref{thm:Poisson-0} for $G = H^1(\A_K,T)$ and $G_0 = H^1(K,T)$,
\[
  \sum_{f \in H^1(K,T)} w(f) = \frac{\size{H^0(K,T)}}{\size{H^0(K,T^*)}} \sum_{g \in H^1(K,T^*)} \hat{w}(g),
\]
where we abuse notation and write $w:H^1(K,T)\to \C$ for the pullback of $w:H^1(\A_K,T)\to \C$ along $i_T$.
\end{theorem}
\begin{proof}
The image $\im(\iota_T)$ is a discrete subgroup of $H^1(\A_K,T)$, and the Poitou-Tate nine-term exact sequence implies $\im(\iota_{T^*})=\im(\iota_T)^{\perp}$. Thus, Poisson summation (Theorem \ref{thm:Poisson-0}) implies that there exists a constant $c$ such that
\[
\int_{f\in \im(\iota_T)} w(f) \, df = c\cdot \int_{g\in \im(\iota_{T^*})} \hat{w}(g)\, dg.
\]
The constant $c$ depends on how the measures $df$ and $dg$ are scaled. We specified a scaling at the beginning, so we cannot assume the two integrals are equal on the nose. We know that $\im(\iota_T)$ and $\im(\iota_{T^*})$ are discrete, so that the above is equivalent to
\[
\sum_{f\in \im(\iota_T)} w(f) = c\cdot \sum_{g\in \im(\iota_{T^*})} \hat{w}(g).
\]
Recalling that $w$ is defined on $H^1(\A_K,T)$ and that the corresponding map $w:H^1(K,T) \rightarrow \C$ is defined to be the pullback along $\iota_T$, we see in particular that $w(f h) = w(f)$ for any $h\in \ker(\iota_T)$. That same statement is true for $T$ replaced by $T^*$, so
\begin{equation} \label{eq:Poisson_w}
    \sum_{f\in H^1(K,T)} w(f) = \frac{c\size{\ker(\iota_{T^*})}}{\size{\ker(\iota_T)}} \sum_{g\in H^1(K,T^*)}\hat{w}(g).
\end{equation}
It now suffices to determine the constant coefficient
\[
c' = \ds \frac{c\size{\ker(\iota_{T^*})}}{\size{\ker(\iota_T)}},
\]
which is independent of $w$. For this, we can simply take for $w$ any function for which either sum in \eqref{eq:Poisson_w} is nonzero.

We note that when $w = \1_L$ is the characteristic function of a compact open box $L = \prod_p L_p$, the left-hand side of \eqref{eq:Poisson_w} is the cardinality of the \emph{Selmer group}
\[
  \Sel(L) = \left\{f \in H^1(K, T) : f \in L_p \, \forall p\right\}.
\]
(If $i_T$ has nontrivial kernel, it is necessarily contained in $\Sel(L)$ for every $L$.)
With our normalizations, the Fourier transform of $w_p = \1_{L_p}$ as a weighting on $H^1(K_p, T)$ is
\[
  \hat w_p = \frac{\size{L_p}}{\size{H^0(K, T)}} \1_{L_p^\perp},
\]
and by the multiplicative definition of the pairing,
\[
  \hat w = \prod_p \hat w_p = \prod_p \frac{\size{L_p}}{\size{H^0(K, T)}} \1_{L^\perp}.
\]
So \eqref{eq:Poisson_w} becomes an identity relating the Selmer groups of $L$ and $L^\perp$:
\[
  \frac{\size{\Sel(L)}}{\size{\Sel(L^\perp)}} = c' \prod_p \frac{\size{L_p}}{\size{H^0(K_p, T)}}.
\]
This identity is none other than the \emph{Greenberg-Wiles formula}
\[
  \frac{\size{\Sel(L)}}{\size{\Sel(L^\perp)}} = \frac{\size{H^0(K,T)}}{\size{H^0(K,T^*)}} \prod_p \frac{\size{L_p}}{\size{H^0(K_p, T)}}
\]
(see \cite[Theorem 2.19]{darmon-diamond-taylor1995} or \cite[Theorem 3.11]{jorza_notes2013}), where the constant coefficient
\[
  c' = \ds \frac{\size{H^0(K,T)}}{\size{H^0(K,T^*)}}
\]
arises out of a global Euler-characteristic computation.
\end{proof}

The use of Theorem \ref{thm:Poisson} separates into two cases:
\begin{enumerate}
    \item If $w$ is \textbf{periodic} with respect to a compact open subgroup of the form
    \[
    \prod_{p\not\in S} H^1_{ur}(K_p,T)
    \]
    for some finite set of places $S$, then $\hat{w}$ has compact support. This implies that the sum over $H^1(K,T^*)$ is secretly a finite sum. In this case, it will suffice to understand the behavior of $\hat{w}(g)$ at individual values of $g$. (We note that this is similar to classical Poisson summation for real-valued functions, where the Fourier transform of a periodic function has discrete support. The extra structure of being periodic with respect to an \emph{open} subgroup is what gives an even smaller support in this case.)
    \item If $w$ is \textbf{not periodic}, then the sum over $H^1(K,T^*)$ is truly an infinite sum. This situation is worth considering, as we can show that it is closely related to other important questions in number theory and arithmetic statistics. While dealing with the infinite sum requires more powerful analytic tools than we will employ in this paper, the behavior of each individual summand $\hat{w}(g)$ will give important insight into certain questions about nonabelian extensions.
\end{enumerate}

We will study several examples of both periodic and non-periodic $w$ functions which are important to open questions in number theory, and in the periodic case we will use these results to prove Theorem \ref{thm:comparisonthm}.

\subsection{Viable local conditions}\label{sec:viable}

Before continuing, we must deal with the circumstance that some local conditions are not satisfiable at all.

\begin{example}\label{ex:GW}
It is known by the Grunwald--Wang theorem that there exist no $C_8$-extensions of $\Q$ such that $2$ is totally inert.

Thus, if $C_8$ carries the trivial Galois action and $\mathcal{L}$ is defined by
\[
L_p = \begin{cases}
\{\text{totally inert coclasses}\} & p=2\\
H^1(\Q_p,C_8) & \text{else},
\end{cases}
\]
Then $H^1_{\mathcal{L}}(\Q,C_8) = \emptyset$.
\end{example}

Wood deals with a similar issue in \cite{wood2009}, which corresponds to the case that $T$ carries the trivial Galois action and the invariant ``$\inv$" is the conductor (or some other fair invariant). In such cases, only the places above $2$ and $\infty$ can cause this type of issue. Wood proves that, as long as at least one extension with the prescribed local conditions exist, then the asymptotic growth rate is as expected.

We make similar definitions to Wood to distinguish these cases:
\begin{definition}
Let $\mathcal{L} = (L_p)$ be a family of subsets $L_p \subseteq H^1(K_p,T)$ and $S$ a finite set of places such that for all $p\not\in S$, $H^1_{ur}(K_p,T) \subseteq L_p$.

We say $\phi = (\phi_p)\in \prod_{p\in S} L_p$ is \textbf{viable} if there exists some $f\in H^1_{\mathcal{L}}(K,T)$ for which $f|_{G_{K_p}} = \phi_p$ for each $p\in S$. Otherwise we say $\phi$ is \textbf{inviable}.

Likewise, we say $\mathcal{L}$ is \textbf{viable} if $H^1_{\mathcal{L}}(K,T) \neq \emptyset$. Otherwise we say $\mathcal{L}$ is \textbf{inviable}.
\end{definition}

From this definition, it is clear that whenever $\mathcal{L}$ is inviable it must be that $H^1_{\mathcal{L}}(K,T) = \emptyset$. Example \ref{ex:GW} is then an example of an inviable family of local conditions.

When $L_p$ is the whole $H^1(K_p, T)$ at almost all places, the search for a global coclass $f$ satisfying the finitely many given local conditions is called \emph{weak approximation} or the \emph{Grunwald problem} in the literature \cite{Wang1950, Neukirch1973, Harari2007, DLAN17, Motte18}. In particular, it is known that there is a finite set $S_0$ of places, depending only on $T$, such that if local conditions are imposed only at a finite set of places disjoint from $S_0$, the Grunwald problem is solvable.

Determining which conditions are viable is tantamount to computing the image of the map
\[
  H^1(K, T) \to \prod_{p \in S} H^1(K_p, T).
\]
By Poitou-Tate duality, this is closely related to the kernel of the map
\begin{equation}\label{eq:loczn_T*}
  H^1(K, T^*) \to \sideset{}{'}\prod_{p \notin S} H^1(K_p, T^*).
\end{equation}
These maps are studied extensively, in various cases, in Neukirch--Schmidt--Wingberg \cite{neukirch-schmidt-wingberg2008}. In the course of proving Theorem \ref{thm:comparisonthm}, we will also prove the following classification of viability:

\begin{theorem}\label{thm:viability}
Let $\mathcal{L} = (L_p)$ be a Frobenian family of local conditions such that for all but finitely many places, $H^1_{ur}(K_p,T)\subseteq L_p$ and $L_p$ is a union of cosets of $H^1_{ur}(K,T)$.

Define the \textbf{viability testing subgroup}
\[
V_{\mathcal{L}} = \{f\in H^1(K,T^*) : f_p \in L_p^{\perp}\text{ for all but finitely many }p\}.
\]
Its orthogonal complement $V_{\mathcal{L}}^\perp \subseteq H^1(\A_K, T)$ has the property that the system $\mathcal{L}=(L_p)$ of local conditions is viable if and only if
\[
  V_{\mathcal{L}}^\perp \intsec \sideset{}{'}\prod_p L_p \neq \emptyset,
\]
where the product is the restricted product with respect to $H^1_{ur}(K_p,T)$.
\end{theorem}

The viability testing subgroup can alternatively be written as a union of dual Selmer sets
\[
V_{\mathcal{L}} = \bigcup_S H^1_{S\mathcal{L}^{\perp}}(K,T^*),
\]
where the union is over all finite sets of places $S$, and $S\mathcal{L}^{\perp}$ is the family of local conditions
\[
\begin{cases}
H^1(K_p,T^*) & p\in S\\
L_p^{\perp} & p\not\in S.
\end{cases}
\]
If $L_p$ is larger, then we generally expect $V_{\mathcal{L}}$ to be smaller. The proof of Theorem \ref{thm:viability} separates naturally into two cases depending on the value of $a_{\inv}(\mathcal{L})$:
\begin{enumerate}
\item Suppose $a_{\inv}(\mathcal{L}) = \infty$, which by \cite[Definition 4.6]{alberts2021} is equivalent to $L_p = H^1_{ur}(K,T)$ for all but finitely many places $p$. This corresponds to the case that $H^1_{\mathcal{L}}(K,T)$ is finite, which follows from the Greenberg-Wiles identity. Then
\begin{align*}
V_{\mathcal{L}} &= \{f\in H^1(K,T^*) : f_p\in H^1_{ur}(K_p,T^*)\text{ for all but finitely many }p\}\\
&=H^1(K,T^*).
\end{align*}
Poitou--Tate duality states that $H^1(K,T^*)$ exactly annihilates $i_T(H^1(K,T))\subset H^1(\A_k,T)$, so it immediately follows that
\[
i_T(H^1(K,T)) = H^1(K,T^*)^{\perp} \cap \sideset{}{'}{\prod_p} L_p.
\]
One side is nonempty if and only if the other is, concluding the proof in this case.
\item Suppose $a_{\inv}(\mathcal{L}) < \infty$. Equivalently, $H^1_{ur}(K,T)\subsetneq L_p$ for a positive proportion of places. This corresponds the the case that $H^1_{\mathcal{L}}(K,T)$ is infinite, and will be proven in the course of proving Theorem \ref{thm:comparisonthm}.
\end{enumerate}

We discuss a few features of viability, but we do not give a full treatment of the structure of $V_{\mathcal{L}}$. Such work would certainly be interesting, but would go well beyond the bounds of this paper.

\subsection{Places that witness inviability}

Viability of local conditions is well-studied in the case that $T$ has the trivial action. In fact, a complete classification of viable local restrictions when $T$ has the trivial action is given in \cite{neukirch-schmidt-wingberg2008,wood2009}.

One of the more useful features of the trivial action case is that inviability is only witnessed by places above $2$ and $\infty$. This helps make viability statements more digestible, and generally speaking number theorists are happy to accept that $2$ and $\infty$ have unique problems not seen by odd primes. However, it turns out that for general Galois modules odd primes \emph{can} be the source of inviable local conditions. See \cite{Wang1950, Neukirch1973, Harari2007, DLAN17, Motte18} for some examples when $L_p = H^1(K_p,T)$ is as large as possible for all but finitely many places.

In the greatest generality, inviability can get particularly bad. There is no guarantee that only finitely many places witness inviability and it is possible for $V_{\mathcal{L}}$ to be an infinite group. For example, in the case that $L_p = H^1_{ur}(K_p,T)$ for all but finitely many $p$ one finds that
\begin{align*}
V_{\mathcal{L}} &= \{f\in H^1(K,T^*) : f\text{ unramified at all but finitely many places}\}\\
&= H^1(K,T^*).
\end{align*}
When $V_{\mathcal{L}}$ is too large, $V_{\mathcal{L}}^{\perp}$ becomes very small and makes the intersection in Theorem \ref{thm:viability} more likely to be empty.

While we do not propose to give a full treatment of viability in this corrigendum, we do consider a case of interest where $V_{\mathcal{L}}$ is necessarily small.

\begin{lemma}\label{lem:V}
Let $T^*$ be a Galois module over $K$ and $\mathcal{L}$ be a family of local conditions such that $L_p$ generates $H^1(K_p,T)$ for all but finitely many places. Then the viability testing subgroup is given by
\[
  V_{\mathcal{L}} = \{f \in H^1(K,T^*) : f_p = 0 \text{ for all but finitely many } p\}.
\]
Moreover, $V_{\mathcal{L}}$ satisfies the following properties:
\begin{enumerate}[(a)]
  \item\label{it:fin} $V_{\mathcal{L}}$ is a finite group,
  \item\label{it:field} $\res^F_{K}(V_{\mathcal{L}}) = \{0\} \subseteq H^1(F, T^*)$, where $F$ is the field of definition of $T^*$ (that is, the smallest field for which $G_F$ acts trivially on $T^*$);
  \item\label{it:S} For each $f\in V_{\mathcal{L}}$ and $p\not\in S_0$, $f_p = 0$, where $S_0$ is the finite set of places at which $T^*$ is ramified.
\end{enumerate}
\end{lemma}

Lemma \ref{lem:V} generalizes the work in \cite{DLAN17} by allowing infinitely many local restrictions. When $T$ has the trivial action, it is known that $S_0$ can be shrunk further to the set of places above $2$ and $\infty$ only. In fact, a complete classification of viable local restrictions when $T$ has the trivial action is given in \cite{neukirch-schmidt-wingberg2008,wood2009}. Unfortunately, this feature is specific to the trivial-action case. The $S_0$ given in Lemma \ref{lem:V}\ref{it:S} cannot be shrunk in general as discussed in \cite{DLAN17}.

In Example \ref{ex:GWoddp} we give examples where inviability is witnessed by any odd prime (i.e., cases for which $S_0$ necessarily contains an odd prime).

\begin{proof}[Proof of Lemma \ref{lem:V}]
Part \ref{it:fin} follows from part \ref{it:field}. Because $V_{\mathcal{L}}\subseteq \ker(\res_K^F)$, we have embedded $V_{\mathcal{L}}$ as a subgroup of $H^1(\Gal(F/K), T^*)$ by the inflation-restriction sequence. Thus, $V_{\mathcal{L}}$ is finite.

We now prove \ref{it:field}. Let $f \in V_{\mathcal{L}}$. By a convenient viewpoint on Galois cohomology (see \cite{odorney2021reflection}, Proposition 4.21), we can view $f$ as a homomorphism
\[
  \sigma_f : G_K \to T^* \rtimes \Aut(T^*)
\]
whose projection on the second factor agrees with the $G_K$-action on $T^*$. That is, $\sigma_f(g) = (f(g),\phi_{T^*}(g))\in T^*\rtimes \Aut(T^*)$. Let $E \supset F$ be the fixed field of the kernel of $\sigma_f$. Note that $E$ is a Galois extension of $K$ containing $F$, which is also Galois, and we have a commutative diagram
\[ 
\begin{tikzcd}
  \Gal(E/K) \rar[hook]{\sigma_f} \dar[two heads] & T^* \rtimes \Aut(T^*) \dar[two heads] \\
  \Gal(F/K) \rar[hook]{\phi_{T^*}} & \Aut(T^*).
\end{tikzcd}
\]
We claim that $E = F$. If not, let $g \in \Gal(E/F) \subseteq \Gal(E/K)$ be a nonidentity element. By the Chebotarev density theorem, there exist infinitely many $p$ such that $\Fr_p(E/K) = g$. But for such $p$, $f_p(\Fr_p)= \sigma_f(\Fr_p) = \Fr_p(E/K) \neq 0 \in T^*$. This means $f_p$ is a nonzero cocycle. Additionally, $\Fr_p(E/K) \in \Gal(E/F)$ implies $\Fr_p(F/K) = 1$. Thus, for an infinite set of these places unramified in $F$, the action of $G_{K_p}$ on $T^*$ is trivial. This implies the set of coboundaries is trivial, so $f_p \neq 0$ as a coclass in $H^1(K_p,T^*)$ for infinitely many places. This contradicts $f \in V_{\mathcal{L}}$, and so proves \ref{it:field}.

Lastly we prove \ref{it:S}. Let $p$ be a place unramified in $F/K$ with $g=\Fr_p(E/K)$. For unramified infinite places, we trivially have $H^1(\Gal(F_\infty/K_\infty),T^*) = 0$, so it suffices to consider finite places. For any other place $\ell$ with $\Fr_\ell(E/K)=\Fr_p(E/K)=g$, the local action of $T^*$is the same at $p$ and $\ell$ so that $H^1(K_p,T^*)\cong H^1(K_\ell,T^*)$. Moreover, $\sigma_f(\Fr_p)$ and $\sigma_f(\Fr_\ell)$ are equal up to conjugation. Thus,
\begin{align*}
f_p(\Fr_p) &= \sigma_f(\Fr_p)\phi_{T^*}(g)^{-1}\\
&= h\sigma_f(\Fr_\ell) h^{-1}\phi_{T^*}(g)^{-1}\\
&= h\sigma_f(\Fr_\ell) h^{-1}\sigma_f(\Fr_\ell) f_\ell(\Fr_\ell).
\end{align*}
This implies $f_p$ and $f_\ell$ are equivalent (up to coboundaries) under the isomorphism of local cohomology groups. By the Chebotarev density theorem, there are infinitely many choices for $\ell$ with $\Fr_\ell(E/K) = g$. Thus, if $f_p$ is not a coboundary, the $f_\ell$ is not a coboundary for infinitely many $\ell$ and it follows that $f\not\in V_{\mathcal{L}}$. Then \ref{it:S} follows from the contrapositive.
\end{proof}

\subsection{Example of inviability at odd places}

A complete classification of viability for general $T$ and $\mathcal{L}$ is outside the scope of this paper. Instead, we provide the following example to demonstrate that inviability can be witnessed by any place when we allow nontrivial actions, which generalizes Example \ref{ex:GW}.

\begin{example}\label{ex:GWoddp}
We will construct an inviable local restriction at an odd prime over $\Q$. The precise example we prove is given below, although with a little more work this construction can be expanded to create a number of other examples.

Let us use the following notation. If $a,b \in \Q^\cross$, denote by $T(a,b)$ the Galois module whose underlying group is $\Z/8\Z$ and such that $\Gal(\bar{\Q}/\Q)$ permutes the four generators $1,3,5,7 \in \Z/8\Z$ in the same manner as the ordered pairs
\[
(\sqrt{a},\sqrt{b}), \quad (-\sqrt{a}, -\sqrt{b}), \quad (\sqrt{a}, -\sqrt{b}), \quad (-\sqrt{a}, \sqrt{b}).
\]
Observe that $T(a,b)$ has field of definition $\Q(\sqrt{a}, \sqrt{b})$. For example, $\mu_8 = T(-1, 2)$, from which we find that the Tate dual of $T(a,b)$ is
\[
\Hom(T(a,b),\mu_8) \cong T(-a, 2b).
\]

\begin{proposition}
Let $\ell$ be an odd prime and consider the Galois module $T = T(-\ell, 2b)$ where $b$ is a quadratic non-residue modulo $\ell$. Let $S_0$ be a finite set of places including $2$, $\ell$, and all places dividing $b$. Then there exists a $\phi\in \prod_{p\in S_0} H^1(\Q_p,T)$ for which the family of local conditions $\mathcal{L}$ on $H^1(\Q,T)$ given by
\[
L_p = \begin{cases}
H^1(K_p,T) & p\not\in S_0\\
\{\phi_p\} & p \in S_0
\end{cases}
\]
is inviable.
\end{proposition}

\begin{proof}
These examples are constructed by finding an element in $V_{\mathcal{L}}$ which is nontrivial at the place $\ell$, which forces $V_{\mathcal{L}}^{\perp}\cap H^1(\Q_\ell,T) \ne H^1(\Q_\ell,T)$. It then suffices to take any $\phi\in H^1(\Q_\ell,T)$ outside of $V_{\mathcal{L}}^{\perp}$.

Take $f\in H^1((\Z/8\Z)^{\times},\Z/8\Z)$ given by the equivalence class of the crossed homomorphism
\[
f(n) = \begin{cases}
0 & n = 1,7\\
4 & n = 3,5.
\end{cases}
\]
We check this is a crossed homomorphism. Notice that since $\im(f) = 4\Z/8\Z$, we have
\[
(n-1) f(m) = 0
\]
for all $n,m\in (\Z/8\Z)^{\times}$, which rearranging implies that
\[
nf(m) = f(m).
\]
Thus it suffices to prove that $f : (\Z/8\Z)^{\times} \to 4\Z/8\Z$ is a homomorphism, which is evident since $\{1,7\} \subset (\Z/8\Z)^\times$ is a subgroup of index $2$.

The cocycle $f$ has the following nice properties:
\begin{itemize}
\item[(i)] $f$ is not a coboundary. This is because all coboundaries are of the form $n\mapsto na - a = (n-1)a$. However, the equations $(7-1)a = 0$ and $(3-1)a = 4$ are not simultaneously solvable modulo $8$.
\item[(ii)] The image of $f$ under restriction to any cyclic subgroup \emph{is} a coboundary. This is trivial for $f|_{\langle 1\rangle} = 0$ and $f|_{\langle 7\rangle} = 0$. Meanwhile,
\begin{align*}
f|_{\langle 3\rangle}(m) &= \begin{cases}
0 & m=1\\
4 & m=3
\end{cases}\\
&= 2m - 2
\end{align*}
and
\begin{align*}
f|_{\langle 5\rangle}(m) &= \begin{cases}
0 & m=1\\
4 & m=5
\end{cases}\\
&= 1m - 1.
\end{align*}
\end{itemize}
We now inflate $f$ along the Galois action on $T^* = T(\ell, b)$ to an element $\inf(f)\in H^1(\Q,T^*)$, which we claim is a nontrivial element of the viability testing subgroup $V_{\mathcal{L}}\le H^1(\Q,T^*)$. Consider that any place $p$ unramified in the field $F = \Q(\sqrt{\ell}, \sqrt{b})$ of definition of $T^*$, as well as any infinite place $p$, has cyclic decomposition group in $F$. Let $C < (\Z/8\Z)^{\times}$ be the cyclic image of $D_p$ under the isomorphism with $\Gal(F/\Q)$. We have $\size{C} = 1$ or $2$. Highlighting the fact that the action of $D_p$ on $T^*$ factors through the quotient $G_\Q \twoheadrightarrow (\Z/8\Z)^\cross$ as the quotient $D_p \twoheadrightarrow C$, we produce a commutative diagram
\[
\begin{tikzcd}
H^1((\Z/8\Z)^{\times},T^*)\dar{\res} \rar{\inf} & H^1(\Q,T^*)\dar{\res_{D_p}}\\
H^1(C,T^*) \rar{\inf} &H^1(D_p,T^*).
\end{tikzcd}
\]
By a diagram chase, we find that $\res(f) = 0$ implies that $\res_{D_p}\inf(f) = 0$ at all places unramified in $F/\Q$. This proves that $\inf(f)\in V_{\mathcal{L}}$.

Meanwhile, by construction, $\ell$ is ramified in $F/\Q$ with $D_\ell = \Gal(F/\Q) \cong (\Z/8\Z)^\cross$, so $\res_{D_\ell}\inf(f) \neq 0$ because $f$ is not a coboundary. Thus $\inf(f)$ is nonzero at at least one place, namely $\ell$; in particular, $f \in V_{\mathcal{L}}\setminus \{0\}$. Let $S_0$ be the finite set of places at which $\inf(f) \neq 0$ (taking $2$, $\ell$, and all primes dividing $b$ is sufficient). Take
\begin{align*}
\phi \in \prod_{p\in S_0} H^1(\Q_p,T)
\end{align*}
that is not orthogonal to $\res_{S_0}(\inf f)$. Then no global coclass $g \in H^1(\Q, T)$ can reduce to $\phi$ at every place in $S_0$, or it would have a globally nontrivial pairing with $\inf f$.
\end{proof}
\end{example}

\section{Multiplicative \texorpdfstring{$w$}{w}}

The examples we will consider come from multiplicative functions:

\begin{definition}
We call a function $w:H^1(\A_K,T)\rightarrow \C$ \textbf{multiplicative} if
\[
w(f) = \prod_p w(f|_{G_{K_p}})\,.
\]
\end{definition}

The discriminant is multiplicative, so the function $w(f)=\disc(f)^{-s}$ is necessarily multiplicative (as is the corresponding example for any admissible ordering $\inv$ as defined in \cite[Definition 4.4]{alberts2019}). Multiplicativity is preserved by the Fourier transform, so that the identification of $H^1(K_p,T)^{\vee}$ with $H^1(K_p,T^*)$ by Tate duality implies
\begin{align*}
\hat{w}(g) &= \prod_p\hat{w}(g|_{G_{K_p}})\\
&= \prod_p \left(\frac{1}{\size{H^0(K_p,T)}}\sum_{f_p\in H^1(K_p,T)}\langle f_p,g|_{G_{K_p}}\rangle w(f_p)\right)
\end{align*}
where $\langle\cdot,\cdot\rangle$ is the Tate pairing and we recall that each point in $H^1(K_p,T)$ was assigned measure $1/|H^0(K_p,T)|$. Note that if $g=0$ is the identity, this is reminiscent of the local series given by Bhargava in the Malle-Bhargava principle \cite{bhargava2010,alberts2019}. Indeed, if $w(f)=\disc(f)^{-s}$ then it follows that
\[
\hat{w}(0) = \prod_p\left(\frac{1}{\size{H^0(K_p,T)}}\sum_{f_p\in H^1(K_p,T)}\disc(f_p)^{-s}\right)
\]
is exactly the local series given by Bhargava (in the case of the trivial action) and the first author (in the case of nontrivial actions). This suggests a reformulation of the Malle-Bhargava principle in terms of Fourier transforms:

\begin{heuristic}[Malle-Bhargava principle on $H^1(K,T)$]\label{heu:MB}
Suppose $w:H^1(\A_K,T)\rightarrow \C$ is a ``reasonable" multiplicative function satisfying the hypotheses of Poisson summation (Theorem \ref{thm:Poisson-0}). Then there exists a positive constant $c>0$ such that
\[
\sum_{f\in H^1(K,T)} w(f) = c\cdot \hat{w}(0) + \text{``lower-order term".}
\]
\end{heuristic}

Following in the footsteps of Cohen-Lenstra \cite{cohen-lenstra1984}, we do not give a definition for a ``reasonable" function; instead we will consider a series of examples in this paper that should be considered ``reasonable". The notion of a ``lower-order term" depends on the context as well, and we don't give an explicit definition in all cases. In the example $w(f)=\disc(f)^{-s}$, $\hat{w}(0)$ is a Dirichlet series whose rightmost pole is at $s=1/a(T)$ of order $b(K,T)$ \cite[Theorem 3.3]{alberts2019}. We then take ``lower-order term" to mean a Dirichlet series whose rightmost pole is at $s \in \{z\in \C : {\rm Re}(z)<1/a(T)\} \cup \{1/a(T)\}$ and if it is at $s=1/a(T)$ then it is of an order $<b(K,T)$. By applying a Tauberian theorem, we see that this expression identifies the sum of coefficients of $\hat{w}(0)$ as the main term of
\[
\sum_{\disc(f)<X} 1,
\]
while the ``lower-order term" converts into an error of a smaller magnitude.

\textbf{Remark:} In the discriminant ordering (or other similar orderings), there are certain cases where the main term really is a constant multiple of $\hat{w}(0)$. This will happen when there exist other $g\in H^1(K,T^*)$ with $\hat{w}(g)$ is on the same ``order of magnitude" as $\hat{w}(0)$, and so must also contribute to the main term.

One easily shows that $|\hat{w}(g)|\le |\hat{w}(0)|$ for all $g\in H^1(K,T^*)$ and all $w$ for which the Fourier transform exists, which suggests that $\hat{w}(0)$ should attain the correct order of magnitude, if not the correct multiplicity. (Note: there is certainly something to prove here, as we have done nothing to show that the infinite sum of $\hat{w}(g)$ does not attain a larger order of magnitude than the individual summands. This is why we merely refer to this statement as a ``heuristic".)

\section{Periodic \texorpdfstring{$w$}{w}}

Suppose $w$ is periodic with respect to a compact open subgroup $Y\le H^1(\A_K,T)$ and satisfies the hypotheses of Theorem \ref{thm:Poisson}. Then $\hat{w}$ is supported on the compact open subgroup $Y^{\perp}$, and
\[
\sum_{g\in H^1(K,T^*)}\hat{w}(g) = \sum_{g\in H^1(K,T^*)\cap Y^{\perp}}\hat{w}(g)
\]
is a sum over the group $H^1(K,T^*)\cap Y^{\perp}$. Discreteness of $H^1(K,T^*)$ and compactness of $Y^{\perp}$ imply that $H^1(K,T^*)\cap Y^{\perp}$ is finite, which makes the summation easier to evaluate.

An important example of a periodic function is the absolute discriminant when $T$ has the trivial action. If
\[
Y=\prod_{p\nmid \infty} \Hom_{ur}(\Q_p,T),
\]
then $|\disc(f+y)| = |\disc(f)|$ for every $y\in Y$, as the discriminant only sees ramification behavior. $Y\le H^1(\A_{\Q},T)$ is necessarily a compact open subgroup in the adelic topology. The function $w(f)=|\disc(f)|^{-s}$ is shown to satisfy the hypotheses of Theorem \ref{thm:Poisson} by \cite{frei-loughran-newton2018}, which they then use to compute the location and orders of poles of the corresponding Dirichlet series.

Our main example of a periodic function is a generalization of the absolute discriminant to the largest class of functions for which the same ideas will work. This follows the ideas in \cite{alberts2019} by allowing nontrivial actions on $T$, allowing restricted local conditions $\mathcal{L}=(L_p)$, and allowing an admissible ordering to be used instead of the usual discriminant:

\begin{proposition}\label{prop:wperiodic}
Let $K$ be a number field, $S$ a finite set of places of $K$, and $T$ a Galois module. If
\begin{enumerate}[(a)]
\item{
$\mathcal{L}=(L_p)$ is a family of subsets $L_p\subset H^1(K_p,T)$ such that for all $p\not\in S$, $L_p$ is closed under translation by $H^1_{ur}(K_p,T)$ and $0\in L_p$;
}
\item{
$\inv$ is an admissible ordering as defined in \cite[Definition 4.4]{alberts2019},
}
\end{enumerate}
then the function $w:H^1(\A_K,T)\rightarrow \C$ defined by
\[
w(f) = \begin{cases}
\mathcal{N}_{K/\Q}(\inv(f))^{-s} & f\in \prod_p L_p\\
0 & f\not\in \prod_p L_p
\end{cases}
\]
satisfies the hypotheses of Theorem \ref{thm:Poisson} and is periodic with respect to the compact open subgroup
\[
Y = Y_S = \prod_{p\not\in S} H^1_{ur}(K_p,T).
\]
\end{proposition}

\begin{proof}
$w$ is periodic with respect to $Y$ by definition, so it suffices to check the hypotheses of Poisson summation.

We first prove that $w$ is $L^1$ on $H^1(\A_K,T)$ for $\Re(s)$ sufficiently large. Define
\[
a(w) = \liminf_{\mathcal{N}_{K/\Q}p\to\infty} \min_{\substack{f\in H^1(K_p,T)\\ \nu_p(\inv(f))\ne 0}} \nu_p(\inv(f)).
\]
Then
\begin{align*}
    \int_{H^1(\A_K,T)} |w(f)| \ df &\ll \prod_p \left(1 + O(\mathcal{N}_{K/\Q}(p)^{-a(w)\Re(s)})\right)\,,
\end{align*}
which converges absolutely for $\Re(s)>1/a(w)$. We use this to prove each of the hypotheses of Poisson summation (Theorem \ref{thm:Poisson-0}) with $G=H^1(\A_K,T)$ and $G_0=\im(\iota_T)$:

\begin{enumerate}[(a)]
    \item{
    Since $w$ is $Y$-periodic, we have
    \[
      |w(x)| = \frac{1}{\mu(Y)}\int_{Y} |w(x + y)|\ dy,
    \]
    where, since $Y$ is open, the Haar measures on $Y$ and $G$ can be taken to coincide. Hence
    \begin{align*}
        \int_{G_0} |w(x+z)|\ dz &= \frac{1}{\mu(Y)} \int_{G_0} \int_{Y} |w(x + y + z)|\ dy\ dz \\
        &= \frac{\size{G_0 \intsec Y}}{\mu(Y)} \int_{Y + G_0} |w(x + z)|\ dz \\
        &\leq \frac{\size{G_0 \intsec Y}}{\mu(Y)} \int_{G} |w(z)| \ dz < \infty,
    \end{align*}
    where we used that, since $Y$ is compact and $G_0$ is discrete, their intersection is finite. Thus the integral defining $w'$ is absolutely convergent.
    }
    
    \item{
    We note that
    \begin{align*}
        \int_{G/G_0} |w'(x)| \ dx &= \int_{G/G_0} \left\lvert\int_{G_0} w(x+y)\ dy\right\rvert \ dx\\
        &\le \int_{G/G_0} \int_{G_0} |w(x+y)|\ dy\ dx\\
        &= \int_{G} |w(x)|\ dx\,,
    \end{align*}
    which converges because $w$ is $L^1$ for $\Re(s)>1/a(w)$, thus $w'$ is $L^1$ for $\Re(s)>1/a(w)$.
    }
    
    \item
    Since $w'$ is $L^1$, its Fourier transform $\widehat{w'}$ is bounded and hence $L^1$ on the compact space $G^{\vee}/G_0^{\perp}$. \qedhere
\end{enumerate}
\end{proof}

This is almost precisely the setup for Theorem \ref{thm:comparisonthm}, where the only condition missing is that $\mathcal{L}$ and $\inv$ are Frobenian. Indeed, Proposition \ref{prop:wperiodic} implies that
\[
\sum_{f\in H^1(K,T)} w(f)
\]
is a Dirichlet series which can be expressed as a finite sum of Euler products. Knowing that $\mathcal{L}$ and $\inv$ are Frobenian implies that the Euler products are of Frobenian factors, which is sufficient for calculating the locations of their poles as is done in \cite[Proposition 2.2]{alberts2019}. A Tauberian theorem is then used to convert the analytic information at the poles into asymptotic information. Additional steps are needed to address viability considerations. We refer to \cite{alberts2019} for some details that are unchanged.

We now begin the proof of Theorem \ref{thm:comparisonthm}. Let $w$ be as in Proposition \ref{prop:wperiodic}. Then
\[
\sum_{f\in H^1(K,T)} w(f) = \frac{|H^0(K,T)|}{|H^0(K,T^*)|} \sum_{g\in H^1(K,T^*)} \hat{w}(g)
\]
is a Dirichlet series. $w$ is periodic with respect to $Y$, which implies $\hat{w}$ has finite support in $H^1(K,T^*)$. The rate of growth of $|H^1_{\mathcal{L}}(K,T;X)|$ is determined by the location and order of the rightmost pole of this series via a Tauberian theorem. It then suffices to show that the rightmost pole occurs at $1/a_{\inv}(\mathcal{L})$ of order $b_{\inv}(\mathcal{L})$.

First consider
\begin{align*}
    \hat{w}(0) = \prod_p \left(\frac{1}{|H^0(K_p,T)|}\sum_{f\in L_p}\mathcal{N}_{K/\Q}(\inv(f))^{-s} \right)\,.
\end{align*}
This is precisely the Frobenian Euler product studied in \cite{alberts2019}, denoted $Q(0,s)$ in Propositions 4.14 and 4.16 of that paper, which was shown to have rightmost pole at $1/a_{\inv}(\mathcal{L})$ of order $b_{\inv}(\mathcal{L})$, that is, a leading term $c(s - 1/a_{\inv}(\mathcal{L})^{-b_{\inv}(\mathcal{L})}$.

For each other term, note that
\begin{align*}
    |\hat{w}(g)| &= \left\lvert\prod_p \left(\frac{1}{|H^0(K_p,T)|}\sum_{f\in L_p}\langle f,g|_{G_{K_p}}\rangle\mathcal{N}_{K/\Q}(\inv(f))^{-s} \right)\right\rvert\\
    &\le \prod_p \left(\frac{1}{|H^0(K_p,T)|}\sum_{f\in L_p}\mathcal{N}_{K/\Q}(\inv(f))^{-\Re(s)} \right)\,,
\end{align*}
which agrees with $\hat{w}(0)$ at $\Re(s)$. Taking a limit as $s\to 1/a_{\inv}(\mathcal{L})$ implies that if $\hat{w}(g)$ has a pole (or a ``fractional pole", i.e.~a branch singularity like $(s-1/a_{\inv}(\mathcal{L}))^{-5/2}$ for example) at $s=1/a_{\inv}(\mathcal{L})$, it is of order at most $b_{\inv}(\mathcal{L})$. Thus, after adding the finitely many other terms $\hat{w}(g)$ to $\hat{w}(0)$, it suffices to show that this rightmost pole of $\hat{w}(0)$ does not cancel out. The remainder of this section will be devoted to proving that the pole does not cancel out, which we accomplish in three steps:

\begin{enumerate}
\item\label{step:part} Partitioning the sum of $\hat{w}(g)$ by cosets of the viability testing subgroup. This step is new, and will prove the ``only if" direction of Theorem \ref{thm:viability}.
\item\label{step:prod} Proving Theorem \ref{thm:correctedmain} in the case that $\inv$ is given by the product of ramified primes outside $S$ and the minimal-index coclasses generate $T$. The ``if" direction of Theorem \ref{thm:viability} follows from these cases.
\item\label{step:remain} Proving the remaining cases of Theorem \ref{thm:correctedmain} by comparing a general admissible invariant to the product of the ramified primes.
\end{enumerate}

\textbf{Remark:} It is likely that step \ref{step:prod} can be done for any ``fair" counting function, with an appropriate analog of the notion of a fair counting function defined in \cite{wood2009}. This step produces a leading term given by a single Euler product, and some notion of fairness will likely preserve this property. Meanwhile, as demonstrated for the discriminant ordering in \cite{wood2009}, some orderings do not produce a single Euler product for the leading constant and necessarily require the work in step \ref{step:remain}.

\subsection{Partitioning by cosets of the viability testing subgroup}
We know that $w$ is periodic with respect to $Y_S$ in the setting of Theorem \ref{thm:correctedmain}, which means $\hat{w}$ is supported on the compact subgroup $Y_S^{\perp}$.  We concluded the sum of $\hat{w}(g)$ is finite because $Y_S^{\perp} \cap H^1(K,T^*)$ is finite. We partition the sum
\[
\sum_{g\in H^1(K,T^*)}\hat{w}(g) = \sum_{g\in R}\left(\sum_{v\in Y_S^{\perp}\cap V_{\mathcal{L}}}\hat{w}(g+v)\right),
\]
where $R$ is a set of representatives for the cosets $(Y_S^{\perp} \cap H^1(K,T^*)) / (Y_S^{\perp} \cap V_{\mathcal{L}})$.

For each $g\in R$ we evaluate the sum of Fourier transforms
\begin{align*}
\lefteqn{\sum_{v\in Y_S^{\perp}\cap V_{\mathcal{L}}}\hat{w}(g+v)} \\
&= \sum_{v\in Y_S^{\perp}\cap V_{\mathcal{L}}}\prod_{p} \left(\frac{1}{|H^0(K_p,T)|}\sum_{f\in L_p}\langle f,(g+v)|_{G_{K_p}}\rangle\mathcal{N}_{K/\Q}(\inv(f))^{-s} \right).
\end{align*}
Each element $v\in V_{\mathcal{L}}$ has $v|_{G_{K_p}}\in L_p^{\perp}$ for all but finitely many places, so for each of those places $\langle f, (g+v)|_{G_{K_p}}\rangle = \langle f, g|_{G_{K_p}}\rangle$ is independent of $v$. Take $\widetilde{S}$ to be the set of places defined by
\[
\widetilde{S} = S\cup \bigcup_{v\in Y_S^{\perp}\cap V_{\mathcal{L}}}\{p : v|_{G_{K_p}} \not\in L_p^{\perp}\},
\]
so that $\widetilde{S}$ is precisely the set of places for which the Euler factor depends on $v$. We factor these places out so that
\begin{align*}
\lefteqn{\sum_{v\in Y_S^{\perp}\cap V_{\mathcal{L}}}\hat{w}(g+v)} \\
&= \sum_{v\in Y_S^{\perp} \cap V_{\mathcal{L}}}\prod_{p\in \widetilde{S}} \left(\frac{1}{|H^0(K_p,T)|}\sum_{f\in L_p}\langle f,(g+v)|_{G_{K_p}}\rangle\mathcal{N}_{K/\Q}(\inv(f))^{-s}\right)\times \\
&\quad {}\times\prod_{p\not\in \widetilde{S}} \left(\frac{1}{|H^0(K_p,T)|}\sum_{f\in L_p}\langle f,g|_{G_{K_p}}\rangle\mathcal{N}_{K/\Q}(\inv(f))^{-s}\right).
\end{align*}
For simplicity, we write the second Euler product as
\[
\frac{1}{\prod_{p\in \widetilde{S}} \hat{w}_p(g)}\hat{w}(g).
\]
Next, we simplify the first Euler product. The number of places contributed to $\widetilde{S}$ by each individual $v$ is finite since $v\in V_{\mathcal{L}}$. Meanwhile, the set of exceptional places $S$ is finite by definition, and
\[
Y_S^\perp\cap V_{\mathcal{L}}\subset Y_S^{\perp} \cap H^1(K,T^*)
\]
is finite by $Y_S^{\perp}$ compact and $H^1(K,T^*)$ discrete. Thus $\widetilde{S}$ is a finite set.

Let $\res_{\widetilde{S}} : H^1(K, T^*) \to \prod_{p\in \widetilde{S}} H^1(K_p, T^*)$ be the product of restriction maps to places $p\in \widetilde{S}$. We multiply the Euler factors over $\widetilde{S}$ together to get
\begin{align*}
\left(\prod_{p\in \widetilde{S}} \frac{1}{|H^0(K_p,T)|}\right)\sum_{f\in \prod_{p\in \widetilde{S}} L_p}\langle f,\res_{\widetilde{S}}(g+v)\rangle \mathcal{N}_{K/\Q}(\inv(f))^{-s},
\end{align*}
Moving the summation over $v$ all the way to the inside gives an inner sum
\begin{align*}
\sum_{v\in Y_S^{\perp}\cap V_{\mathcal{L}}}\langle f, \res_{\widetilde{S}}(g)+\res_{\widetilde{S}}(v)\rangle &= \langle f, \res_{\widetilde{S}}(g)\rangle \sum_{v\in Y_S^{\perp}\cap V_{\mathcal{L}}}\langle f,\res_{\widetilde{S}}(v)\rangle\\
&=\begin{cases}
\langle f, \res_{\widetilde{S}}(g)\rangle|Y_S^{\perp}\cap V_{\mathcal{L}}| & f\in \res_{\widetilde{S}}(Y_S + V_{\mathcal{L}}^{\perp})\\
0 & \text{else}.
\end{cases}
\end{align*}
\textbf{Remark:} Here we see that the sum over $v$ of the Euler factors at $\widetilde{S}$ plays the same role as the sum over $\mathcal{E}(C)$ in \cite{wood2009}. Whether this sum is zero or not detects inviability.

Putting it all together, we conclude that
\begin{align*}
\sum_{v\in V_{\mathcal{L}}}\hat{w}(g+v) &= |Y_S^{\perp}\cap V_{\mathcal{L}}|\left(\sum_{f\in \res_{\widetilde{S}}(Y_S + V_{\mathcal{L}}^{\perp}) \cap \prod_{p\in \widetilde{S}} L_p} \langle f, g\rangle \mathcal{N}_{K/\Q}(\inv(f))^{-s}\right)\\
&\quad {} \times\left(\prod_{p\in\widetilde{S}} \frac{1}{\hat{w}_p(g)|H^0(K_p,T)|}\right)\hat{w}(g).
\end{align*}
Consider that
\[
\res_{\widetilde{S}}(Y_S+V_{\mathcal{L}}^\perp)\cap \prod_{p\in \widetilde{S}}L_p = \emptyset
\]
if and only if
\[
(Y_S+V_{\mathcal{L}}^\perp)\cap \prod_{p}L_p= \emptyset
\]
by pulling back along $\res_{\widetilde{S}}$. By assumption, $\prod_p L_p$ is closed under translation by $Y_S$, so this set is nonempty if and only if
\[
V_{\mathcal{L}}^{\perp} \cap \prod_p L_p = \emptyset.
\]
As this fact is independent of the representative $g$, emptiness of this set implies $\mathcal{L}$ is inviable and proves the ``only if'' direction of Theorem \ref{thm:viability}.

\subsection{Ordering by the product of the ramified primes}
We consider the specific ordering
\[
\ram_S(f) = \prod_{\substack{p\not\in S\\f|_{I_p}\ne 1}} \Nm_{K/\Q} (p).
\]
This is an example of a ``fair" counting function given in \cite{wood2009}, and we can prove Theorem \ref{thm:correctedmain} for this ordering with similar strength to Wood's results. If $a_{\ram_S}(\mathcal{L}) = \infty$, Theorem \ref{thm:correctedmain} follows from bounding the counting function by a finite Selmer group in the Greenberg-Wiles identity. Otherwise, by definition $a_{\ram_S}(\mathcal{L}) = 1$ as primes can only occur in $\ram_S$ with exponent $1$.

Suppose $g\in Y_S^{\perp} \setminus V_{\mathcal{L}}$, so that $g|_{G_{K_p}}\not\in L_p^{\perp}$ for infinitely many places and $g$ is unramified at all places $p\not\in S$. The Chebotarev density theorem and $\mathcal{L}$ being Frobenian implies that $g|_{G_{K_p}}\not\in L_p^{\perp}$ for a positive proportion of places. This implies that, for a positive proportion of places $p$, there exists an $f\in L_p$ such that $\langle f, \res_p(g)\rangle \ne 1$. However, the fact that $g\in H^1_{ur}(K_p,T^*)$ for these places implies that $f\not\in H^1_{ur}(K_p,T)$ as these subgroups annihilate each other. Because all ramified coclasses have weight $1=a_{\ram_S}(\mathcal{L})$ under the ordering $\ram_S$, we have proven the following:

\begin{lemma}
  For a positive proportion of places, there exists an $f\in L_p^{[1]}$ such that $\langle f, \res_p(g)\rangle \ne 1$ (where, as in \cite{alberts2021}, we set $L_p^{[m]} = \{f_p\in L_p : \nu_p(\inv(f_p))= m\}$). Using this fact, we will prove that the order of the singularity of $\hat{w}(g)$ is strictly less than for $\hat{w}(0)$.
\end{lemma}

The order of the singularity is given by the $b$-invariant, which is the average value of coefficients of $|p|^{-a_{\ram_S}(\mathcal{L})s} = |p|^{-s}$ for $|p|\le x$ as $x\to \infty$ (see \cite[Definition 4.6]{alberts2021}). This coefficient for $\hat{w}(g)$ is given by
\[
\sum_{f\in L_p^{[1]}}\langle f,\res_p(g)\rangle |p|^{-s} \le \sum_{f\in L_p^{[1]}}|p|^{-s}.
\]
By our choice of $g$, this inequality is necessarily strict for a positive proportion of places. Taking the average of both sides over $|p|\le x$ as $x\to \infty$, this implies
\[
\substack{\displaystyle\text{order of singularity of}\\\displaystyle\hat{w}(g)\text{ at }s=1/a_{\ram_S}(\mathcal{L})} < b_{\ram_S}(K,T) = \substack{\displaystyle\text{order of singularity of}\\\displaystyle\hat{w}(0)\text{ at }s=1/a_{\ram_S}(\mathcal{L}).}
\]
In particular, this implies that for any $g\in Y_S^{\perp}\setminus V_{\mathcal{L}}$, the singularity of $\hat{w}(g)$ does not cancel with the singularity of $\hat{w}(0)$. Since $\hat{w}$ has finite support in $Y_S^{\perp}$, the sum over all $g\not\in V_{\mathcal{L}}$ cannot cancel with the singularity of $\hat{w}(0)$ either. Thus, the order of the singularity at $s=1/a_{\ram_S}(\mathcal{L})$ is determined by those $g$ in the identity coset $V_{\mathcal{L}}$, which is given by
\begin{align*}
\sum_{v\in Y_S^{\perp} \cap V_{\mathcal{L}}}\hat{w}(v) & = |Y_S^{\perp} \cap V_{\mathcal{L}}|\left(\sum_{f\in \res_{\widetilde{S}}(Y_S + V_{\mathcal{L}}^{\perp}) \cap \prod_{p\in \widetilde{S}} L_p} \mathcal{N}_{K/\Q}(\inv(f))^{-s}\right)\\
&\quad \times\left(\prod_{p\in \widetilde{S}} \frac{1}{\hat{w}_p(0)|H^0(K_p,T)|}\right)\hat{w}(0).
\end{align*}
All factors are positive at $s=1/a_{\ram_S}(\mathcal{L}) = 1$, so the order of this singularity agrees with that of $\hat{w}(0)$.

Applying a Tauberian theorem, this proves the asymptotic in Theorem \ref{thm:correctedmain} as long as
\begin{itemize}
\item[(a)] $\inv = \ram_S$, and
\item[(b)] $V_{\mathcal{L}}^{\perp}\cap \prod L_p \ne \emptyset$.
\end{itemize}
The nonzero asymptotic in particular shows that (b) implies viability, proving the remaining direction of Theorem \ref{thm:viability}.

\subsection{Extending to the general case}

The final step for proving Theorem \ref{thm:correctedmain} is the case that
\begin{itemize}
\item $\inv \ne \ram_S$, and
\item $\mathcal{L}$ is viable (or equivalently $V_{\mathcal{L}}^{\perp} \cap \prod_{p\in S} L_p \ne \emptyset$, as we have now proven Theorem \ref{thm:viability}).
\end{itemize}
This step proceeds similarly to the original proof, where we produce a lower bound with the desired order of magnitude in order to prove that the main term cannot be cancelled out.

We are assuming that $\mathcal{L}$ is viable, so we choose some $f_0\in H^1_{\mathcal{L}}(K,T)$. Define $\mathcal{L}(f_0) = (L(f_0)_p)$ to be the family of local conditions given by
\[
L(f_0)_p = \begin{cases}
\{0\} & p\in S\\
L_p - f_0|_{G_{K_p}} & p\not\in S.
\end{cases}
\]
For $p\not\in S$, $L_p$ is a union of cosets of $H^1_{ur}(K_p,T)$. This property is preserved in the construction of $L(f_0)_p$, and $f_0\in H^1_{\mathcal{L}}(K_p,T)$ implies $f_0|_{G_{K_p}}\in L_p$ so that $0\in L(f_0)_p$. Thus the corresponding $w$ function is still $Y_S$-periodic. $\mathcal{L}(f_0)$ is certainly viable, as $0\in L(f_0)_p$ for all places $p$.

The map $x\mapsto x+f_0$ gives an injection
\[
H^1_{\mathcal{L}(f_0)}(K,T) \hookrightarrow H^1_{\mathcal{L}}(K,T).
\]
This map preserves the $p$-part of the invariant, $\inv$, at every place unramified in $f_0$. As a result, the invariant can change by at most a bounded amount. This implies that there exists some constant $C_1 > 0$ depending only on $f_0$ such that
\[
H^1_{\mathcal{L}}(K,T;X) \ge H^1_{\mathcal{L}(f_0)}(K,T;C_1X).
\]
Thus it suffices to prove $H^1_{\mathcal{L}(f_0)}(K,T;X)$ has the expected order of magnitude.

Define the subfamily $\mathcal{L}(f_0)^{\rm min} = (L(f_0)_p^{\rm min})$ of $\mathcal{L}(f_0)$ to be the family of local conditions
\[
L(f_0)_p^{\rm min} = \begin{cases}
L(f_0)_p & p\in S\\
L(f_0)_p^{[a_{\inv}(\mathcal{L})]}\cup H^1_{ur}(K,T) & p\not\in S
\end{cases}
\]
Once again, for all but finitely many places $L(f_0)_p^{[a_{\inv}(\mathcal{L})]}$  is a union of cosets of $H^1_{ur}(K_p,T)$. One still has $0\in H^1_{\mathcal{L}(f_0)^{[a_{\inv}(\mathcal{L})]}}(K,T)$, so this family is viable as well. Thus, we have a containment of nonempty Selmer sets
\[
H^1_{\mathcal{L}(f_0)^{[a_{\inv}(\mathcal{L})]}}(K,T) \subseteq H^1_{\mathcal{L}(f_0)}(K,T).
\]

By construction, it follows that for any $f\in H^1_{\mathcal{L}(f_0)^{[a_{\inv}(\mathcal{L})]}}(K,T)$ the invariant is given by
\begin{align*}
\inv(f) &= \prod_{p\in S} p^{\nu_p(\inv(f))}\cdot\prod_{\substack{p\not\in S\\f|_{I_p}\ne 1}} p^{a_{\inv}(\mathcal{L})}\\
&\ge C_2\ram_S(f)^{a_{\inv}(\mathcal{L})}.
\end{align*}
for some positive constant $C_2$ determined by the minimal values of $\nu_p\inv$ on the finite sets $L(f_0)_p^{[a_{\inv}(\mathcal{L})]}$ for the finitely many $p\in S$.

By combining the two embeddings, we produce a lower bound
\[
|H^1_{\inv,\mathcal{L}}(K,T;X)| \ge \left|H^1_{\ram_S,\mathcal{L}(f_0)^{[a_{\inv}(\mathcal{L})]}}\left(K,T;\left(C_1C_2^{-1} X\right)^{1/a_{\inv}(\mathcal{L})}\right)\right|.
\]
Step 2 applies to the lower bound, so we can produce the order of magnitude. One sees directly from the definitions that
\[
  a_{\ram_S}(\mathcal{L}(f_0)^{[a_{\inv}(\mathcal{L})]}) = 1 \quad \text{and} \quad b_{\ram_S}(K,\mathcal{L}(f_0)^{[a_{\inv}(\mathcal{L})]}) = b_{\inv}(K,\mathcal{L}),
\]
as ``$\inv$" at all but finitely many places is preserved under both embeddings of Selmer sets and $L(f_0)_p^{[a_{\inv}(\mathcal{L})]}$ only contains elements of index $a_{\inv}(\mathcal{L})$ or unramified elements (i.e. index $\infty$). Thus, we produce a lower bound of the desired magnitude proving that the main terms do not cancel.

\subsection{Proof of Corollary \ref{cor:surj}}

Corollary \ref{cor:surj} is then a direct consequence of Theorem \ref{thm:comparisonthm}. We summarize the proof here, but we refer to \cite{alberts2019}, Theorem 5.3 for the details.

\begin{proof}[Proof of Corollary \ref{cor:surj}]
We remark that if $f\sim f'$ under the coboundary relation, then there exists a $t\in T$ such that $(f*\pi)(G_K)=t(f*\pi)(G_K)t^{-1}$. Therefore the $T$-conjugacy class of the image is preserved under the coboundary relation, so in particular surjectivity is preserved under the coboundary relation and the numerator in the corollary statement is well-defined.

That the limit exists follows from the following inclusion-exclusion argument. For a subgroup $H \leq G$, let $\mathcal{L}(H\cap T) = (L_p\cap H^1(K_p,H\cap T))_p$. (Here, and in what follows, we suppress pullbacks and pushforwards along the inclusion $H \intsec T \hookrightarrow T$.) Then, as long as there exists at least one $f_H\in H^1_{\mathcal{L}}(K,T)$ such that $(f_H*\pi)(G)\le H$, it follows that
\begin{align*}
\{f\in H^1_{\mathcal{L}}(K,T;X) : (f*\pi)(G_K) \le H\} = f_H * H^1_{\mathcal{L}(H\cap T),\inv_{f_H}}(K,H\cap T;X)
\end{align*}
where $\inv_{f_H}(f)=\inv(f_H*f)$.

It is clear that $\mathcal{L}(H\cap T)$ and $\inv_{f_H}$ inherit the hypotheses of Theorem \ref{thm:comparisonthm} from $\mathcal{L}$ and ``$\inv$" away from the places ramified in $f_H$, except for viability. So Theorem \ref{thm:comparisonthm} implies that
\begin{align*}
&\Size{\{f\in H^1_{\mathcal{L}}(K,T;X) : (f*\pi)(G_K) \le H\}}\\
&\sim c_{\inv_{f_H}}(K,\mathcal{L}(H\cap T)) X^{1/a_{\inv_{f_H}}(\mathcal{L}(H\cap T))}(\log X)^{b_{\inv_{f_H}}(\mathcal{L}(H\cap T))-1},
\end{align*}
where $c_{\inv_{f_H}}$ is positive or $0$ according as viability holds or not for $H$.

Let $\mu_G$ be the M\"obius function on the poset of $T$-conjugacy classes of subgroups $H\le G$, ordered by inclusion up to conjugation, so that
\begin{align*}
&\frac{\Size{\{f\in H^1_{\mathcal{L}}(K,T;X) : f*\pi\text{ surjective}\}}}{|H^1_{\mathcal{L}}(K,T;X)|}\\
&\sim \frac{\ds\sum_H \mu_G(H)c_{\inv_{f_H}}(K,\mathcal{L}(H\cap T)) X^{1/a_{\inv_{f_H}}(\mathcal{L}(H\cap T))}(\log X)^{b_{\inv_{f_H}}(\mathcal{L}(H\cap T))-1}}{c_{\inv}(K,\mathcal{L}) X^{1/a_{\inv}(\mathcal{L})}(\log X)^{b_{\inv}(\mathcal{L})-1}}.
\end{align*}
As the summation is a finite sum of functions of the form $c X^a (\log X)^{b-1}$, the limit certainly exists (we know the limit is not infinite because the original ratio is bounded between $0$ and $1$).

The proofs of (a), (b), and (c) are referred to \cite{alberts2021}. The extra reasoning we need to provide here is that viability implies the existence of a \emph{surjective} coclass, not just any coclass:
\begin{enumerate}
\item[(a)] In this case, \cite{alberts2021} bounds the numerator below by
\[
|H^1_{\mathcal{L}_{\pi}}(K,T;X)|,
\]
where the family $\mathcal{L}_\pi$ is given by
\[
(L_\pi)_p = \begin{cases}
\{0\} & p\in U\\
L_p & \text{else}
\end{cases}
\]
and $U$ is chosen to be a set of places unramified in $\pi$ such that $\{\pi(\Fr_p) : p\in U\} = G$. Without loss of generality, $U$ can be chosen disjoint from $\widetilde{S}$, so that Theorem \ref{thm:viability} implies the viability of $\mathcal{L}_\pi$. Theorem \ref{thm:correctedmain} then gives an asymptotic lower bound.
\item[(b)] Part (c) with the product-of-ramified-primes invariant implies that there exists some $f$ for which $f*\pi$ is surjective. Replacing $\pi$ with $f*\pi$ to define the action, we then apply part (a). Any reference to viability need only be made in parts (a) and (c).
\item[(c)] In this case, it is proven that all but the leading term of the inclusion-exclusion tend to zero in the limit. Thus, the numerator and denominator are necessarily asymptotic to each other. As it was not required to reference viability, this part implies that viability implies the existence of a surjective coclass.
\end{enumerate}


In order to prove part (ii), it suffices to show that $\disc_\pi$ is admissible and Frobenian. This proven in \cite[Lemma 5.1]{alberts2019}.
\end{proof}

\section{Non-periodic \texorpdfstring{$w$}{w}}

In all of the following examples, $w(f)$ will be a function of a complex variable $s\in \C$ so that
\[
\sum_{f\in H^1(K,T)} w(f)
\]
is a Dirichlet series, similar to the periodic example in Proposition \ref{prop:wperiodic}. However, these examples will not be periodic and we will find that Theorem \ref{thm:Poisson} decomposes this Dirichlet series as an infinite sum of Euler products.

We can still reference the Malle-Bhargava principle, which suggests that sum should have the same ``order of magnitude" as $\hat{w}(0)$. Although this is not necessarily true for arbitrary infinite sums (even sums of Euler products), we are able to give evidence that this is close to the truth.

\subsection{Infinitely many local conditions}

\begin{proposition}\label{prop:wnonperiodic}
Let $K$ be a number field, $S$ a finite set of places of $K$, and $T$ a Galois module. Define the function $w:H^1(\A_K,T)\rightarrow \C$ by
\[
w(f) = \begin{cases}
\mathcal{N}_{K/\Q}(\inv(f))^{-s} & f\in \prod_p L_p\\
0 & f\not\in \prod_p L_p,
\end{cases}
\]
where
\begin{enumerate}[(a)]
\item{
$\mathcal{L}=(L_p)$ is a family of subsets $L_p\subset H^1(K_p,T)$ such that for all $p\not\in S$, $H^1_{ur}(K_p,T)\subset L_p$,
}

\item{
$\inv$ is an admissible ordering as defined in \cite{alberts2019}.
}
\end{enumerate}
Then $w$ satisfies the hypotheses of Theorem \ref{thm:Poisson}.
\end{proposition}
\begin{proof}
Simply note that $|w(f)|$ is bounded above by $\bar{w}(f) = |\mathcal{N}_{K/\Q}(\inv(f))|^{-s}$, the case that occurs when $L_p = H^1(K_p, T)$ is as large as possible. Note that $\bar{w}(f)$ is $Y_S$-periodic and satisfies the hypotheses of Theorem \ref{thm:Poisson} by Proposition \ref{prop:wperiodic}. The remainder of the proof proceeds exactly like the proof of Proposition \ref{prop:wperiodic}.
\end{proof}

The difference between Propositions \ref{prop:wnonperiodic} and \ref{prop:wperiodic} is the removal of the translation invariance requirement on the family of allowable local conditions $\mathcal{L}$. The only condition we impose is that all but finitely many places allow for any unramified behavior, which we need to assume in order to avoid specifying the splitting types of infinitely many unramified primes (if we did specify these, it might happen that no coclass satisfies all the local conditions!)

In this setting, Theorem \ref{thm:Poisson} tells us that
\[
\sum_{f\in H^1(K,T)}w(f) = \frac{|H^0(K,T)|}{|H^0(K,T^*)|} \sum_{g\in H^1(K,T^*)} \hat{w}(g)
\]
is an infinite sum of Euler products. In order to apply a Tauberian theorem to get asymptotic information with these kinds of local restrictions, some more work needs to be done to uniformly control the order of $\hat{w}(g)$.

We present the simplest example of this phenomenon below:

\begin{example}
Let $K=\Q$, $T=C_2$ have the trivial action, and for each prime $p \nmid 2\infty$, choose a subset $L_p$ with
\[
  H^1_{ur}(\Q_p,C_2) \subsetneq L_p \subsetneq H^1(\Q_p,C_2);
\]
necessarily $\size{L_p} = 3$. This amounts to choosing one of the two ramified quadratic extensions of $\Q_p$; the easiest choice, which we will make though it is not essential, is
\[
  L_p = H^1_{ur}(\Q_p,C_2) \union \{h_p\}, \quad h_p \text{ corresponds to } \Q_p(\sqrt{p}).
\]
Complete the family $(L_p)$ of local conditions by setting
\[
  L_2 = H^1_{ur}(\Q_2,C_2), \quad L_\infty = H^1(\R,C_2).
\]
Then counting cohomology classes satisfying the local conditions $L_p$, with respect to discriminant, amounts to counting squarefree integers $D \equiv 1$ mod $4$ such that, for every $p|D$, $D/p$ is a square modulo $p$: a cute question in arithmetic statistics, which we have not seen before.

If we define $w:H^1(\Q,C_2)\rightarrow \C$ as in Proposition \ref{prop:wnonperiodic} with $\mathcal{L}$ as above and the usual discriminant $\disc$, then it follows (since $C_2$ is its own Tate dual) that
\[
\sum_{f\in H^1(\Q,C_2)} w(f) = \sum_{g\in H^1(\Q,C_2)}\hat{w}(g)\,.
\]
For an individual $g$, we see that
\begin{align*}
    \hat{w}(g) &= \prod_{p\nmid 2\infty}\left(\frac{1}{|H^0(\Q_p,C_2)|}\sum_{f_p\in L_p}\langle f_p,g|_{G_{\Q_p}}\rangle p^{-\nu_p(\disc(f_p))s}\right)\\
    &=\prod_{p\mid \disc(g)}\left(\frac{\langle h_p,g|_{\Q_p}\rangle}{2}p^{-s}\right) \prod_{p\nmid 2\disc(g)\infty}\left(1+\frac{\langle h_p, g|_{\Q_p}\rangle}{2}p^{-s}\right).
\end{align*}
The first product is of the form $\pm 2^{-\omega(\disc(g))}|\disc(g)|^{-s}$, while the second product is like an $L$-function. If $\chi_g$ is the character of $g$, then
\[
\prod_{p\nmid 2\disc(g)\infty}\left(1+\frac{\langle h_p, g|_{\Q_p}\rangle}{2}p^{-s}\right) = \sum_{2\nmid n} \chi_g(n) 2^{-\omega(n)}\mu^2(n) n^{-s},
\]
which can be shown to be (conditionally) convergent and nonzero at $s=1$ unless $g=0$. If $g=0$, it has a pole of order $1/2$ at $s=1$. In order to conclude that $\hat{w}(0)$ really is the main term, we would need to show that (up to the signs of the summands)
\[
\sum_{g\in H^1(\Q,C_2)}2^{-\omega(\disc(g))}\size{\disc(g)}^{-s} \sum_{2\nmid n} \chi_g(n) 2^{-\omega(n)}\mu^2(n) n^{-s}
\]
is of the same or smaller ``order of magnitude". This amounts to computing the moments of $L$-function look-alikes, and while in principle this is an accessible question we will not address it in this paper.
\end{example}

\bibliographystyle{abbrv}
\bibliography{BrandonStylePlusBib/BAreferencesV2019}
\end{document}